\documentclass[a4paper,12pt]{article}
\usepackage[utf8]{inputenc}
\usepackage[colorlinks=true, allcolors=blue]{hyperref}
\usepackage{enumitem}
\usepackage{lslformat, lslmath, lslthm, lslname, lslabbrv, lslhat}
\usepackage{mathtools}
\usepackage[lmargin=25mm, rmargin=25mm, top=25mm, bottom=30mm, paperwidth=210mm, paperheight=11in]{geometry} 
\usepackage{natbib}

\numberwithin{equation}{section}
\setlist[enumerate]{label=(\roman*)}

\newcommand{\senorm}[1]{\norm{#1}_{\psi_1}}
\newcommand{\sgnorm}[1]{\norm{#1}_{\psi_2}}

\newcommand{\sigsp}{\sig_{\rm SP}}
\newcommand{\sigspU}{\sig_{{\rm SP}+}}
\newcommand{\sigspL}{\sig_{{\rm SP}-}}
\newcommand{\sigspUL}{\sig_{{\rm SP}\pm}}

\newcommand{\sigsg}{\sig_{\rm SG}}
\newcommand{\sigsgU}{\sig_{{\rm SG}+}}
\newcommand{\sigsgL}{\sig_{{\rm SG}-}}
\newcommand{\sigsgUL}{\sig_{{\rm SG}\pm}}

\newcommand{\ala}{\abs{\la}}

\newcommand{\updated}[1]{{#1}}

\begin{document}

\author{Lasse \Leskela \and Ian \Valimaa}
\date{16 July 2026}
\title{Sub-Poisson distributions: Concentration inequalities, optimal variance proxies, and closure properties\thanks{This version of the article has been accepted for publication, after peer review but is
not the Version of Record and does not reflect post-acceptance improvements, or any corrections. The Version of Record
is published in Sankhya A
and is available online at: \url{http://dx.doi.org/10.1007/s13171-026-00444-x}.}}
\maketitle

\abstract{We introduce a nonasymptotic framework for sub-Poisson distributions with moment generating function dominated by that of a Poisson distribution. At its core is a new notion of optimal sub-Poisson variance proxy, analogous to the variance parameter in the sub-Gaussian setting. This framework allows us to derive a Bennett-type concentration inequality without boundedness assumptions and to show that the sub-Poisson property is closed under key operations including independent sums and convex combinations, but not under all linear operations such as scalar multiplication. We derive bounds relating the sub-Poisson variance proxy to sub-Gaussian and sub-exponential Orlicz norms. Taken together, these results unify the treatment of Bernoulli and Poisson random variables and their signed versions in their natural tail regime.}

{\bf Keywords:} Bennett inequality, Bernstein inequality, Orlicz norm, sub-Poisson tail, concentration of mass
%{\bf MSC classes:}  60E05, 60E15

\section{Introduction}

Understanding and quantifying the fluctuations of random variables is a central theme in modern statistics and machine learning.
Classical tools rely heavily on Gaussian-like assumptions:
centered random variables whose moment generating function is bounded by that of a Gaussian,
so that a variance-like parameter captures tail decay.
Such \emph{sub-Gaussian} variables underpin a wide range of concentration inequalities,
which are key in analysing sums of independent variables, regression estimators,
and high-dimensional statistical procedures \citep{Boucheron_Lugosi_Massart_2013,Buldygin_Kozachenko_1980,Vershynin_2018,Wainwright_2019}.

For sparse, discrete, or Poisson-type data---such as counts, Bernoulli trials, and network edges---sub-Gaussian bounds
may be overly conservative.
The \emph{sub-Poisson} framework provides a natural extension of the sub-Gaussian paradigm to variables whose tails resemble those of a Poisson distribution, providing a variance-proxy interpretation that better reflects the behaviour of such data.
The concept of sub-Poisson tails has---mainly implicitly---appeared earlier in the context of
Bennett-type concentration inequalities for sums of random variables that are either fully bounded
\citep{Bennett_1962,Vershynin_2018,Zheng_2018} or bounded from above \updated{almost surely}
\citep{Barlow_Sherlock_2025,BenHamou_Boucheron_Ohannessian_2017,Bentkus_2002,Boucheron_Lugosi_Massart_2013,Pinelis_2014,Wellner_2017}.

Our contribution is to develop a systematic framework and unified taxonomy for sub-Poisson distributions,
introducing optimal variance proxies that quantify tail behaviour in analogy with the sub-Gaussian case.
In particular, we derive a sharp, nonasymptotic Bennett-type concentration inequality that holds without any boundedness assumptions.  In addition, we provide explicit bounds linking the sub-Poisson variance proxy to classical sub-Gaussian and sub-exponential Orlicz norms, and give a detailed characterisation of closure properties under standard transformations. The framework unifies the treatment of Bernoulli, Poisson, and signed discrete distributions and is particularly useful in statistical estimation, network analysis, and risk assessment, where Poisson-like variability arises naturally \citep{Avrachenkov_Dreveton_Leskela_2025+,Barlow_Sherlock_2025,Matias_Rebafka_Villers_2018,Sawarni_Pal_Barman_2023,Valimaa_Leskela_2025}, including settings with signed count data \citep{Karlis_Ntzoufras_2006,Lishamol_Veena_2022,Rave_Kauermann_2024}. In this way, the familiar variance-based theory of Gaussian variables extends naturally to discrete and Poisson-type phenomena, preserving interpretability while capturing realistic tail behaviour.

The rest of the paper is organised as follows. Section~\ref{sec:Definitions} introduces sub-Poisson distributions and their terminology. Section~\ref{sec:ConcentrationInequalities} presents the main concentration inequalities, while Section~\ref{sec:OptimalVarianceProxies} develops the optimal sub-Poisson variance proxy. Section~\ref{sec:ClosureProperties} examines closure properties under standard operations, and Section~\ref{sec:RelatedClasses} relates sub-Poisson variables to bounded, sub-Gaussian, and sub-exponential classes. Finally, Section~\ref{sec:Examples} provides concrete examples and computes their optimal variance proxies.  Some technical results needed in the proofs are postponed into Appendix~\ref{sec:AuxiliaryResults}.

\section{Sub-Poisson distributions}
\label{sec:Definitions}

All random variables are real-valued unless otherwise indicated.
A random variable $X$ is called \emph{integrable} if $\E \abs{X} < \infty$,
and \emph{square-integrable} if $\E X^2 < \infty$.
An integrable random variable is called \emph{centered} if $\E X = 0$.
The moment generating function of $X$ is denoted by $M_X(\lambda) = \E e^{\lambda X}$,
defined as an extended real number in $[0,\infty]$ for all $\lambda \in \R$. We write
\begin{equation}
\label{eq:PoissonLMGF}
 \phi(\lambda)
 \weq e^\lambda - 1 - \lambda
\end{equation}
for the logarithmic moment generating function of a centered Poisson variable with unit variance.
A random variable $X$, \updated{or its law}, is called
%A random variable $X$ and its probability distribution is called
\begin{itemize}
\item \emph{upper sub-Poisson} if there exists $\sigma^2 \ge 0$ such that
\begin{equation}
\label{eq:SPU}
\E e^{\lambda (X - \E X)} \le e^{\sigma^2 \phi(\lambda)}
\qquad \text{for all $\la \ge 0$},
\end{equation}
\item \emph{lower sub-Poisson} if $-X$ is upper sub-Poisson,
\item \emph{sub-Poisson} if it is both upper and lower sub-Poisson,
or equivalently
\begin{equation}
 \label{eq:SP}
 \E e^{\la (X - \E X)}
 \wle e^{\sig^2 \phi(\ala)}
 \qquad \text{for all $\la \in \R$}.
\end{equation}
\end{itemize}

These definitions mirror their classical sub-Gaussian counterparts, obtained 
by replacing the function $\phi(\la)$ with $\frac12 \la^2$ in \eqref{eq:SPU}--\eqref{eq:SP}.
They also admit a natural interpretation
via stochastic order theory \citep{Leskela_2010,Leskela_Ngo_2017,Muller_Stoyan_2002,Shaked_Shanthikumar_2007}.
For random variables
we denote $A \lemgf B$ and 
say that $A$ is less than $B$ in the
\emph{moment generating function order}
if 
$\E e^{\la A} \le \E e^{\la B}$ for all $\la \ge 0$.
Denote $\tilde X = X - \E X$
and let $\tilde Y = Y - \E Y$ where $Y$ is Poisson distributed with variance $\sig^2$.
Then $X$ is
\begin{itemize}
\item upper sub-Poisson if and only if $\tilde X \lemgf \tilde Y$,
\item lower sub-Poisson if and only if $-\tilde X \lemgf \tilde Y$.
\end{itemize}
Therefore, the upper (resp.\ lower) sub-Poisson property means that
the upper (resp.\ lower) tail of $X-\E X$
is dominated in the $\lemgf$-order by
the upper tail of $Y - \E Y$.
In particular, the random variable $X$ is sub-Poisson with variance proxy $\sig^2$
if and only if both tails of $X-\E X$
are dominated by the upper tail of the centered Poisson random variable $Y - \E Y$
with variance $\sig^2$.

\begin{remark}
An alternative definition,
appearing in \cite{BenHamou_Boucheron_Ohannessian_2017} corresponds to
replacing $\phi(\ala)$ by $\phi(\la)$ on the right side of \eqref{eq:SP}.
This amounts to comparing the lower tail of a random variable $X$ to a lower instead of an upper Poisson tail,
and implicitly requires $X$ to be almost surely bounded from below.
In contrast, our symmetric formulation compares both tails to upper Poisson tails,
which aligns more naturally with Bennett- and Bernstein-type bounds,
parallels the established framework for sub-Gaussian and sub-exponential variables,
and guarantees that all sub-Gaussian random variables are sub-Poisson.
\end{remark}

%\begin{remark}
%For an alternative approach, Wellner \cite{Wellner_2017} defines a family Bennett--Orlicz norms parameterised by a parameter $L > 0$.  This is closely related to the sub-Poisson assumption.
%The fact not a single norm, but rather a family of norms is needed, is reflected in the fact that
%the space of sub-Poisson random variables is not a vector space.
%Maybe mention here also \cite{Sawarni_Pal_Barman_2023}.
%\end{remark}

\section{Concentration inequalities}
\label{sec:ConcentrationInequalities}

The following result summarises three concentration inequalities
for upper sub-Poisson random variables.  
The first \eqref{eq:Bennett} is a Bennett-type inequality,
%,
and the latter \eqref{eq:Bernstein1}--\eqref{eq:Bernstein2} correspond to Bernstein's inequalities.
Unlike in Bennett's original inequality \citep{Bennett_1962,Vershynin_2018,Zheng_2018}
for bounded random variables, and its various extensions for random variables that are
bounded above almost surely \citep{Barlow_Sherlock_2025,BenHamou_Boucheron_Ohannessian_2017,Bentkus_2002,Boucheron_Lugosi_Massart_2013,Pinelis_2014}, the following inequalities are valid without any boundedness assumptions.

When both tails of $X$ are sub-Poisson, i.e.\ $X$ is sub-Poisson with variance proxy $\sig^2$,
we obtain two-sided bounds
by noting that
\[
 \pr( \abs{\updated{X  - \E X}} \ge t )
 \wle \pr( \updated{X  - \E X} \ge t ) + \pr( -(\updated{X - \E X}) \ge t)
\]
and applying Proposition~\ref{the:Chernoff} to $X$ and $-X$.
When combined with Proposition~\ref{the:IndependentSum},
we obtain bounds for independent sums of sub-Poisson variables.

\begin{proposition}
\label{the:Chernoff}
If $X$ is upper sub-Poisson with variance proxy \updated{$\sig^2 > 0$},
then for all $t \ge 0$,
\begin{align}
 \pr( \updated{X - \E X} \ge t)
 &\wle e^{-\sig^2}\left(\frac{e\sig^2}{\sig^2 + t}\right)^{\sig^2 + t}\label{eq:Bennett} \\
 &\wle \exp\left(-\frac{t^2/2}{\sig^2+t/3}\right) \label{eq:Bernstein1}\\
 &\wle \exp\left(-\left(\frac{t^2}{4\sig^2}\wedge\frac{3t}{4}\right)\right). \label{eq:Bernstein2}
\end{align}
\end{proposition}

\updated{
\begin{remark}\label{rem:ChernoffConstantCase}
In the case $\sig^2=0$, Proposition~\ref{the:Variance} implies that $\Var(X)=0$ and hence $X=\E X$ almost surely.
\end{remark}}

\begin{proof}[Proof of Proposition~\ref{the:Chernoff}]
Fix a number \updated{$\la \ge 0$}.  Markov's inequality implies that
\begin{align*}
 \pr(\updated{X - \E X}\ge t)
 \wle \pr(e^{\la (\updated{X - \E X})} \ge e^{\la t})
 &\wle e^{-\la t}\E e^{\la (\updated{X - \E X})} \\
 &\wle \exp\left( -\la t+\sig^2(e^{\la} - 1-\la) \right).
\end{align*}
Differentiation shows that the right side above is minimised with $\la=\log(1+t/\sig^2)\ge 0$. 
Inequality \eqref{eq:Bennett} follows by substituting this value to the above inequality, and simplifying the outcome.

Observe next that $k! \ge 2 \cdot 3^{k-2}$ for all $k \ge 2$.
This implies that
%Bernstein's inequality \eqref{eq:Bernstein1} follows from
\[
 e^{\la} - 1-\la
 \weq \sum_{k=2}^{\infty}\frac{\la^k}{k!}
 \wle \sum_{k=2}^{\infty}\frac{\la^k}{2\cdot 3^{k-2}}
 %\weq \frac{\la^2}{2}\sum_{k=0}^{\infty}\frac{\la^k}{3^k}
 \weq \frac{\la^2/2}{1-\la/3}
 \qquad \text{for $0 \le \la < 3$}.
\]
It follows that
\begin{align*}
 e^{-\sig^2} \left(\frac{e\sig^2}{\sig^2 + t}\right)^{\sig^2 + t} 
 \nhquad = \ \exp\left(\inf_{\la\ge 0} \left(-\la t+\sig^2(e^{\la} - 1-\la)\right)\right)
 \wle \exp\left(\inf_{\updated{0 \leq \la < 3}} f(\la) \right),
\end{align*}
where $f(\la) = -\la t+\frac{\sig^2\la^2/2}{1-\la/3}$.
Let us reparameterise this function as $f(\la) = 9 \sig^2 g(\al)$,
where $\al = 1-\la/3 \in \updated{(0,1]}$ and $g(\al) = \frac{(1-\al)^2}{2\al} - \frac{t}{3\sig^2} (1-\al)$.
Denoting $z = \frac{t}{3\sig^2}\ge 0$, we see that
\[
 g(\al)
 \weq \left(\frac12 +z\right)\al + \frac12 \al^{-1} - 1 - z.
\]
Differentiation shows that $g$ is strictly convex on \updated{$(0,1]$} and attains its minimum at $\al_* = (1+2z)^{-1/2}$.
By substituting this value into the above equality, we find that
\[
g(\al_*)
\weq (1+2z)^{1/2} - (1 + z)
\weq \frac{ (1+2z)-(1+z)^2}{(1+2z)^{1/2}+1+z}
\weq -\frac{z^2}{(1+2z)^{1/2}+1+z}.
\]
By estimating now the concave function $(1+2z)^{1/2}$ from above by its linear approximation $1+z$, we find that
\begin{align*}
 g( \alpha_* )
 \wle - \frac{z^2}{2+2z}
 \weq - \frac{t^2/ (9 \sig^4)}{2 + 2 t/(3 \sig^2)}
 \weq - \frac{t^2/ (9 \sig^2)}{2 \sig^2 + (2/3)t}.
\end{align*}
Then
\[
 \inf_{\updated{0 \le \la < 3}} f(\la)
 \wle 9 \sig^2 g(\al_*)
 \wle - \frac{t^2}{2 \sig^2 + (2/3)t}
 \wle - \frac{t^2/2}{\sig^2 + t/3},
\]
from which we conclude that \eqref{eq:Bernstein1} is valid.

Inequality \eqref{eq:Bernstein2} simply follows from
\[
 \exp\left(-\frac{t^2/2}{\sig^2+t/3}\right)
 \wle \exp\left(-\frac{t^2/2}{(2\sig^2)\vee (2t/3)}\right)
 \weq \exp\left(-\left(\frac{t^2}{4\sig^2}\wedge\frac{3t}{4}\right)\right).
\]
\end{proof}

\section{Optimal variance proxies}
\label{sec:OptimalVarianceProxies}

For any random variable $X$ with a finite mean,
the \emph{optimal sub-Poisson variance proxy}
is defined by
\begin{equation}
 \label{eq:ProxySP}
 \sigsp^2(X)
 \weq \sup_{\la \ne 0} \frac{\log \E e^{\la (X - \E X)} }{\phi(\ala)},
\end{equation}
where $\phi$ is given by \eqref{eq:PoissonLMGF}.
Similarly, the \emph{optimal upper and lower sub-Poisson variance proxies} 
are defined by setting
\begin{equation}
 \label{eq:ProxySPU}
 \sigspU^2(X)
 \weq \sup_{\la > 0} \frac{\log \E e^{\la (X - \E X)} }{\phi(\la)}
\end{equation}
and $\sigspL^2(X) = \sigspU^2(-X)$.

\begin{proposition}
\label{the:ProxySP}
An integrable random variable $X$ is sub-Poisson
(resp.\ upper sub-Poisson)
if and only if
$\sigsp^2(X)$
(resp.\ $\sigspU^2(X)$) is finite,
and in this case $\sig^2 = \sigsp^2(X)$ 
(resp.\ $\sigspU^2(X)$)
is the smallest
number for which \eqref{eq:SP}
(resp.\ \eqref{eq:SPU}) 
holds.
\end{proposition}
\begin{proof}
We observe that a number $\sig^2 \ge 0$ satisfies \eqref{eq:SP}
if and only if 
\begin{equation}
 \label{eq:Proxy1}
 \sig^2 \ge \frac{\log \E e^{\la (X - \E X)}}{\phi(\ala)}
 \quad \text{for all $\la \ne 0$}.
\end{equation}
In light of definition \eqref{eq:ProxySP}, we see that \eqref{eq:Proxy1} is equivalent
to $\sig^2 \ge \sigsp^2(X)$.
We conclude that the set of numbers $\sig^2 \ge 0$ that satisfy \eqref{eq:SP}
is equal to $[\sigsp^2(X), \infty)$, and the latter interval is nonempty if and only if
$\sigsp^2(X) < \infty$.

Similarly, we see that
$\sig^2 \ge 0$ satisfies \eqref{eq:SPU}
if and only if \eqref{eq:Proxy1} holds for all $\la > 0$.
An analogous reasoning then confirms the claim for the upper sub-Poisson properties.
\end{proof}

%\subsection{Variance proxy is at least as large as the variance}

The following result shows that any (upper/lower) sub-Poisson variance proxy
is always greater than the variance of the random variable.
In particular, any (upper/lower) sub-Poisson random variable
must have a finite second moment.

\begin{proposition}
\label{the:Variance}
$\Var(X) \le \sigspUL^2(X)\le \sigsp^2(X)$ for any integrable random variable $X$.
\end{proposition}

\begin{proof}
Since  $\sigspL^2(X) = \sigspU^2(-X)$
and $\sigsp^2(X) = \sigspU^2(X) \vee \sigspL^2(X)$,
we only need to verify that $\Var(X) \le \sigspU^2(X)$.
We assume that $\sig^2 = \sigspU^2(X)$ is finite, to exclude trivial cases.
Denote $\ttX = X - \E X$ and $\ttY = Y - \E Y$ where $Y$ is Poisson distributed with mean $\sig^2$.
By Proposition~\ref{the:ProxySP}, 
we see that $M_{\ttX}(\la) \le M_{\ttY}(\la)$ for all $\la \ge 0$.
Because $\E \ttX = \E \ttY = 0$, it follows that
\begin{equation}
 \label{eq:VarianceProof1}
 M_{\ttX}(\la) - 1 - \la \E \ttX
 \wle M_{\ttY}(\la) - 1 - \la \E \ttY
 \qquad \text{for all $\la \ge 0$}.
\end{equation}
If $M_{\ttX}$ were twice differentiable at 0, we could apply a second-order Taylor expansion
to conclude from \eqref{eq:VarianceProof1} that
$\Var(X) \le \Var(Y)$. However, such an expansion may not be justified, since $M_{\ttX}$
might not be differentiable in any neighbourhood of zero.
To overcome this, we approximate $\ttX$ by truncated random variables $\ttX_n = \ttX 1(\abs{\ttX} \le n)$.
Fix a number $\la \ge 0$.
Because the function $x \mapsto e^{\la x} - 1 - \la x$ is increasing on $[0,\infty)$ and decreasing on $(-\infty,0]$,
we see that 
\[
 e^{\la \ttX_n} - 1 - \la \ttX_n
 \wle e^{\la \ttX} - 1 - \la \ttX
\]
almost surely.
By taking expectations
and applying \eqref{eq:VarianceProof1},
we conclude that
\begin{equation}
 \label{eq:VarianceProof2}
 M_{\ttX_n}(\la) - 1 - \la \E \ttX_n
 \wle M_{\ttY}(\la) - 1 - \la \E \ttY
 \qquad \text{for all $\la \ge 0$}.
\end{equation}
Because $\ttX_n$ is a bounded variable, we know that $M_{\ttX_n}(\la)$ is analytic on the full real line,
with $M_{\ttX_n}''(0) = \E \ttX_n^2$. Because \updated{$M_{\ttX_n}(0) = 1$ and $M'_{\ttX_n}(0) = \E \ttX_n$},
it follows that
\[
 \lim_{\la \downto 0} \frac{M_{\ttX_n}(\la) - 1 - \la \E \ttX_n}{\la^2/2}
 \weq \E \ttX_n^2.
\]
A similar expansion holds for the centered Poisson moment generating function
$M_{\ttY}(\la) = e^{\sig^2(e^\la-1-\la)}$.
Now by 
dividing both sides of \eqref{eq:VarianceProof2} by $\la^2/2$ and taking $\la \downto 0$,
we conclude that
\[
 \E \ttX_n^2
 \wle \E \ttY^2.
\]
By noting that \updated{$\ttX_n^2 \uparrow \ttX^2$} as $n \to \infty$, it follows by the monotone convergence theorem
that
$
 \E \ttX^2 \le \E \ttY^2,
$
which means that $\Var(X) \le \Var(Y) = \sig^2$.
\end{proof}

As a corollary of Proposition~\ref{the:Variance},
we obtain the following characterisation of degenerate random variables.

\begin{proposition}
\label{the:Zero}
The following are equivalent for any integrable random variable $X$:
\[
\text{(i) } X = \E X \ \text{a.s.}, \quad
\text{(ii) } \sigsp^2(X) = 0, \quad
\text{(iii) } \sigspU^2(X) = 0, \quad
\text{(iv) } \sigspL^2(X) = 0.
\]
\end{proposition}
\begin{proof}
(i)$\implies$(ii). If $X = \E X$ almost surely, then $\E e^{\la(X-\E X)} = 1$ for all $\la$,
so that $\sigsp^2(X) = 0$ due to \eqref{eq:ProxySP}.

(ii)$\implies$(iii)\&(iv). Immediate from $\sigsp^2(X) = \sigspU^2(X) \vee \sigspL^2(X)$.

(iii)$\implies$(i).  If $\sigspU^2(X) = 0$, then $\Var(X) = 0$ by Proposition~\ref{the:Variance}, so that $X = \E X$ almost surely.

(iv)$\implies$(i).  Similarly as the proof of (iii)$\implies$(i).
\end{proof}

\section{Closure properties}
\label{sec:ClosureProperties}

\subsection{Convexity}

The spaces of centered sub-Gaussian, centered sub-Poisson, and centered real-valued
random variables defined on a probability space $(\Omega, \cA, \pr)$ are ordered
(see Proposition~\ref{the:SubGaussianComparisonQualitative}) by
\[
 L^{\rm SG}_0(\pr)
 \subset L^{\rm SP}_0(\pr)
 \subset L^1_0(\pr).
\]
The space $L^1_0(\pr)$ is a Banach space equipped with the standard $L^1$-norm.
It is also known that $\sigsg(X)$ defines a norm making $L^{\rm SG}_0(\pr)$
a Banach space \cite[Theorem 1]{Buldygin_Kozachenko_1980}.
In contrast, the space of centered sub-Poisson random variables $L^{\rm SP}_0(\pr)$ does not enjoy such linear structure.
For example, it is not closed under scalar multiplication (Example~\ref{exa:ScaledSkellam}).
Nevertheless, $L^{\rm SP}_0(\pr)$ shares some weaker closure properties with the space $L^p_0(\pr)$ for $0 < p < 1$.
Namely, the following result shows that $L^{\rm SP}_0(\pr)$ is a convex and balanced\footnote{A subset $U$ of a vector space is called \emph{balanced} if $a x \in U$ for all $x \in U$ and $\abs{a} \le 1$.}
subset of $L^1_0(\pr)$.

\begin{proposition}
\label{the:Convexity}
The functional $\sigsp^2 \colon L^{\rm SP}_0(\pr) \to [0,\infty]$
is convex and balanced %on $L_0(\pr)$, in the sense
in the sense that for any $X,Y \in L^1_0(\pr)$:
\begin{enumerate}
\item[\updated{\rm (i)}] \label{the:Convexity:Convex} $\sigsp^2( (1-a) X + aY) \le (1-a) \sigsp^2(X) + a \sigsp^2(Y)$ for $0 \le a \le 1$.
\item[\updated{\rm (ii)}] \label{the:Convexity:Balanced} $\sigsp^2(a X) \le a^2 \sigsp^2(X) \le \sigsp^2(X)$ for $\abs{a} \le 1$.
\end{enumerate}
\end{proposition}
\begin{proof}
(i)
Let $a \in (0,1)$.
\Holder's inequality 
$\E AB \le (\E A^p)^{1/p} (\E B^q)^{1/q}$ applied
to $A = e^{\la (1-a)X }$ and $B = e^{\la a Y }$
with
$p = \frac{1}{1-a}$ and $q = \frac{1}{a}$
implies that
\begin{align*}
 \E e^{\la ( (1-a)X + a Y )}
 \weq \E e^{\la (1-a)X } e^{\la a Y}
 \wle \left( \E e^{\la X } \right)^{1-a} \left( \E e^{\la Y} \right)^{a}
 %&\wle e^{(1-a) \sigsp^2(X) \phi(\la) } e^{a \sigsp^2(Y) \phi(\la)}
\end{align*}
By recalling the fundamental characterisation of the optimal sub-Poisson variance proxy
(Proposition~\ref{the:ProxySP}), it follows that
\[
 \log \E e^{\la ( (1-a)X + a Y )}
 \wle (1-a) \sigsp^2(X) \phi(\ala) + a \sigsp^2(Y) \phi(\ala).
\]
Hence (i) follows by \eqref{eq:ProxySP}.

(ii)
Assume that $\abs{a} \le 1$.
Then
\begin{align*}
 \phi(\abs{a\la})
 \weq \sum_{k \ge 2} \frac{\abs{a}^k \ala^k}{k!}
 \wle a^2 \sum_{k \ge 2} \frac{\ala^k}{k!}
 \weq a^2 \phi(\ala),
\end{align*}
so that by \eqref{eq:ProxySP},
\[
 \sigsp^2(aX)
 \weq \sup_{\la \ne 0} \frac{\log \E e^{\la aX}}{\phi(\ala)}
 \wle \updated{\sup_{\la \ne 0}}\frac{\sigsp^2(X) \phi(\abs{a\la})}{\phi(\ala)}
 \wle a^2 \sigsp^2(X).
\]
\end{proof}

\subsection{Independent sums}

\begin{proposition}
\label{the:IndependentSum}
For any independent integrable random variables $X_1,\dots,X_n$:
\begin{align}
 \label{eq:IndependentU}
 \sigspUL^2(\sum_i X_i) &\wle \sum_i \sigspUL^2(X_i), \\
 %\label{eq:IndependentL}
 %\sigspL^2(\sum_i X_i) &\wle \sum_i \sigspL^2(X_i), \\
 \label{eq:Independent}
 \sigsp^2(\sum_i X_i) &\wle \sum_i \sigsp^2(X_i).
\end{align}
\end{proposition}
\begin{proof}
Let $S=\sum_i X_i$. Then by independence, $\E e^{\la (S-\E S)} = \prod_i \E e^{\la (X_i-\E X_i)}$.
Hence \eqref{eq:IndependentU} for $\sigspU^2$ follows by noting that
\begin{align*}
 \sigspU^2(S)
 \weq \sup_{\la > 0} \sum_i \frac{\updated{\log\,} \E e^{\la (X_i-\E X_i)}}{\phi(\la)} 
 &\wle \sum_i \sup_{\la > 0} \frac{\updated{\log\,}\E e^{\la (X_i-\E X_i)}}{\phi(\la)}
 \weq \sum_i \sigspU^2(X_i).
\end{align*}

Inequality \eqref{eq:IndependentU} \updated{for} $\sigspL^2$ follows by applying 
\eqref{eq:IndependentU} to $-S = \sum_i (-X_i)$.
Inequality \eqref{eq:Independent} follows from \eqref{eq:IndependentU} by 
noting that $\sigsp^2(S) = \sigspU^2(S) \vee \sigspL^2(S)$.
\end{proof}

\subsection{Absolute values}

\begin{proposition}
\label{the:AbsoluteValue}
If $X$ is sub-Poisson,
then $\abs{X-\E X}$ is upper sub-Poisson with
optimal variance proxy 
$\sigspU^2( \abs{X-\E X} ) \le 2 \sigsp^2(X)$.
\end{proposition}
\begin{proof}
Denote $Y = X-\E X$.
Fix $\la \ge 0$.
Observe that
\begin{align*}
 \E e^{\la (\abs{Y} - \E \abs{Y})}
 \weq e^{- \la \E \abs{Y}} \E e^{\la \abs{Y}}
 \weq e^{- \la \E \abs{Y}} \big( 1 + \la \E \abs{Y} + \E \phi(\la \abs{Y}) \big).
\end{align*}
Because $1+t \le e^t$ for all $t$, it follows that
\begin{equation}
 \label{eq:AbsoluteValue1}
 \E e^{\la (\abs{Y} - \E \abs{Y})}
 \wle 1 + e^{- \la \E \abs{Y}} \E \phi(\la \abs{Y})
 \wle 1 + \E \phi(\la \abs{Y}).
\end{equation}
Note also that
$\phi(\abs{y}) = \phi(y) 1(y \ge 0) + \phi(-y) 1(y < 0)
\le \phi(y) + \phi(-y) = e^y + e^{-y} - 2$.
It follows that
\begin{equation}
 \label{eq:AbsoluteValue2}
 1 + \E \phi(\la \abs{Y})
 \wle -1 + M_Y(\la) + M_Y(-\la)
 \wle -1 + 2 e^{\sigsp^2(X) \phi(\la)}.
\end{equation}
The claim follows by combining \eqref{eq:AbsoluteValue1}--\eqref{eq:AbsoluteValue2}
and applying the inequality
$-1+2t \le t^2$.
\end{proof}

\section{Related classes of random variables}
\label{sec:RelatedClasses}

\subsection{Bounded random variables}

The following result confirms that bounded random variables are sub-Poisson,
and provides upper bounds for the optimal sub-Poisson variance proxy.

\begin{proposition}
\label{the:SPBounded}
\leavevmode
\begin{enumerate}
\item[\updated{\rm(i)}] If $a \le X \le b$ almost surely, then $\sigsp^2(X) \le \frac{(b-a)^2}{4}$.
\item[\updated{\rm(ii)}] If $0 \le X \le 1$ almost surely, then $\sigsp^2(X) \le \E X$.
\item[\updated{\rm(iii)}] If $X \le 1$ almost surely, then $\sigspU^2(X) \le \E X^2$.
\item[\updated{\rm(iv)}] If $\abs{X} \le 1$ almost surely, then $\sigsp^2(X) \le \E X^2$.
\item[\updated{\rm(v)}] If $X,\xi$ are independent and such that $X$ is centered and sub-Poisson,
and $\abs{\xi} \le 1$ almost surely, then $\xi X$ is sub-Poisson with $\sigsp^2(\xi X) \le \sigsp^2(X)$.
\end{enumerate}
\end{proposition}

\begin{remark}
\label{rem:BoundedCentered}
Proposition~\ref{the:Variance} indicates that
the conclusions in Proposition~\ref{the:SPBounded}:(iii)--(iv) hold as equality when $\E X = 0$.
\end{remark}

\begin{proof}[Proof of Proposition~\ref{the:SPBounded}]
(i) Assume that $a \le X \le b$ almost surely.
A classical inequality often called Hoeffding's lemma
\cite[Inequality (4.16)]{Hoeffding_1963} implies that
$X$ is sub-Gaussian with
$\sigsg^2(X) \le \frac{(b-a)^2}{4}$.
The claim follows because $\sigsp^2(X) \le \sigsg^2(X)$
by Proposition~\ref{the:SubGaussianComparisonQuantitative}.

(ii) Denote $a = \E X$.
The convexity of the exponential function implies that
$e^{\la X} \le (1-X) e^0 + X e^\la$. By taking expectations
and applying the inequality $1+x \le e^x$, 
we find that
$
 \E e^{\la X}
 \le 1 + a (e^\la-1)
 \le \exp \left( a (e^\la-1) \right).
$
Hence $\E e^{\la (X-\E X)} \le e^{ a \phi(\la)}$,
and due to $\phi(\la) \le \phi(\ala)$ (Lemma~\ref{the:cosh}), we conclude that
$\E e^{\la (X-\E X)} \le e^{ a \phi(\ala)}$ for all $\la \in \R$.
By recalling \eqref{eq:ProxySP}, this implies that $\sigsp^2(X) \le a$.

(iii)  Fix $\la \ge 0$. Denote $m(\lambda)=\log\E e^{\la(X-\E X)}$. The inequality $1+t\le e^t$ implies that
\[
m(\la)
\weq -\la \E X+\log\E e^{\la X}
\weq -\la\E X + \log\left(1 + \la\E X+\E \phi(\la X)\right)
\wle \E\phi(\la X).
\]
Define $f(x)=\phi(x)/x^2$ for all $x \neq 0$ and $f(0)=1/2$.
%so that $m(\la)\le \E\la^2X^2f(\la X)$. 
After verifying that $f$ is increasing, we find that $f(\la X)\le f(\la)$ almost surely.  Hence,
\[
 \phi(\la X)
 \weq (\la X)^2 f(\la X)
 \wle (\la X)^2 f(\la)
 \weq X^2 \phi(\la)
\] 
almost surely. By taking expectations, it follows that
$\E\phi(\la X) \le \phi(\la) \E X^2$.
We conclude that $m(\la) \le \phi(\la) \E X^2$, confirming that $\sigspU^2(X)\le\E X^2$.

It remains to show that $f$ is increasing. 
Because $f(x)=\sum_{k=2}^{\infty}x^{k-2}/k!$
has a power series representation with an infinite radius of convergence, it suffices to show that $f'(x)\geq 0$ for all $x\neq 0$. Define $g(x)=x(e^{x}+1)-2 (e^{x}-1)$. Differentiation shows that
\[
\begin{split}
f'(x)
&\weq \frac{x^2(e^{x}-1)-2x(e^{x}-1-x)}{x^4}
\weq \frac{x^2(e^{x}+1)-2x (e^{x}-1)}{x^4}
\weq \frac{g(x)}{x^3}, \\
g'(x)
&\weq (e^{x}+1) + x e^{x}-2e^{x}
\weq 1-(1-x)e^{x} 
\wge 1-e^{-x}e^{x}
\weq 0.
\end{split}
\]
Since $g(0)=0$, we conclude that $f'(x)=g(x)/x^3\ge 0$ for all $x\neq 0$.

(iv) Assume that $\abs{X} \le 1$ almost surely.
Then $X \le 1$ almost surely, and (iii) implies that
$\sigspU^2(X) \le \E X^2$.
Furthermore, $-X \le 1$ almost surely, and (iii) implies that
$\sigspL^2(X) = \sigspU^2(-X) \le \E X^2$.
The claim follows by recalling that $\sigsp^2(X) = \sigspU^2(X) \vee \sigspL^2(X)$.

(v) We note that $\xi X$ is centered, and the independence of $\xi$ and $X$ implies that
% Independence argument needed? e.g. By Theorem 3.11 in Kallenberg second edition, $\E e^{\la \xi X} = \E (\E e^{\la s X})_{s=\xi}$ almost surely.
% Theorem 3.11 implies the property used here rather quickly: If A is \xi-measurable and $1_A=g(\xi)$ (Doob's representation), this theorem implies that $\E 1_A(\E f(s,X))_{s=\xi}=\E (g(s)\E f(s,X))_{s=\xi}=\E (\Eg(s) f(s,X))_{s=\xi}=\E g(\xi)f(\xi,X)=\E 1_Af(\xi,X)$, i.e., $\E (f(\xi,X)\cond\xi)=(\E f(s,X))_{s=\xi}$ almost surely.
% But is it trivial enough?
\[
 \E ( e^{\la \xi X} \cond \xi )
 \wle e^{\sigsp^2(X) \phi(\abs{\xi \la})}
 \wle e^{\sigsp^2(X) \phi(\abs{\la})}
\]
almost surely. Therefore, $\E e^{\la \xi X} \le e^{\sigsp^2(X) \phi(\abs{\la})}$,
and $\sigsp^2(\xi X) \le \sigsp^2(X)$ due to \eqref{eq:ProxySP}.
\end{proof}

\subsection{Distributions with sub-Gaussian tail mass}

Let us recall the classical definition of sub-Gaussian tails.
An integrable real-valued random variable $X$
is called \emph{sub-Gaussian} (resp.\ \emph{upper sub-Gaussian},
\emph{lower sub-Gaussian})
%is called \emph{upper sub-Gaussian}
if for some number $\sig^2 \ge 0$,
\begin{equation}
 \label{eq:SG}
 \E e^{\la (X - \E X)} \wle e^{\sig^2 \la^2/2 }
\end{equation}
holds for all $\la \in \R$ (resp.\ for all $\la \ge 0$, for all $\la \le 0$).

\begin{proposition}
\label{the:SubGaussianComparisonQualitative}
Every (upper/lower) sub-Gaussian random variable is (upper/lower) sub-Poisson.
\end{proposition}

We prove Proposition~\ref{the:SubGaussianComparisonQualitative} as a corollary
of a quantitative characterisation of associated optimal variance proxies.
For any random variable $X$ with a finite mean,
the \emph{optimal sub-Gaussian variance proxy} is defined by
\begin{equation}
 \label{eq:ProxySG}
 \sigsg^2(X)
 \weq \sup_{\la \ne 0} \frac{\log \E e^{\la (X - \E X)} }{\la^2/2}.
\end{equation}
%We see that $X$ is sub-Gaussian if and only if $\sigsg^2(X) < \infty$,
%and in this case $\sigsg^2(X)$ is the smallest number $\sig^2$ such that \eqref{eq:SG} holds.
The \emph{optimal upper and lower sub-Gaussian variance proxies} 
are defined by
\begin{equation}
 \label{eq:ProxySGU}
 \sigsgU^2(X)
 \weq \sup_{\la > 0} \frac{\log \E e^{\la (X - \E X)} }{\la^2/2}.
\end{equation}
and
$\sigsgL^2(X) = \sigsgU^2(-X)$.

\begin{proposition}
\label{the:SubGaussianComparisonQuantitative}
$\sigsp^2(X) \le \sigsg^2(X)$ and 
$\sigspUL^2(X) \le \sigsgUL^2(X)$
for all integrable real-valued random variables $X$.
\end{proposition}
\begin{proof}
Observe that
\[
 \phi(\la)
 \weq \sum_{k \ge 2} \frac{\la^k}{k!}
 \wge \frac12 \la^2
 \qquad \text{for all $\la \ge 0$}.
\]
Hence
\[
 \sigsp^2(X)
 \weq \sup_{\la \ne 0} \frac{\log \E e^{\la (X - \E X)} }{\phi(\ala)}
 \wle \sup_{\la \ne 0} \frac{\log \E e^{\la (X - \E X)} }{\frac12 \la^2}
 \weq \sigsg^2(X).
\]

By repeating the above argument with the supremum restricted to $\la > 0$, we find that
$\sigspU^2(X) \le \sigsgU^2(X)$.  The inequality $\sigspL^2(X) \le \sigsgL^2(X)$ follows by noting that
$\sigspL^2(X) = \sigspU^2(-X)$
and 
$\sigsgL^2(X) = \sigsgU^2(-X)$.
\end{proof}

\begin{proof}[Proof of Proposition~\ref{the:SubGaussianComparisonQualitative}]
Recall by Proposition~\ref{the:ProxySP} that $X$ is sub-Poisson if and only if
$\sigsp^2(X) < \infty$.
Analogously, we have the classical characterisation \citep{Buldygin_Kozachenko_1980}
that $X$ is sub-Gaussian if and only if 
$\sigsg^2(X) < \infty$.
Proposition~\ref{the:SubGaussianComparisonQuantitative} then implies that
any sub-Gaussian random variable is sub-Poisson.

The claims for upper/lower sub-Poisson properties
follow from Proposition~\ref{the:SubGaussianComparisonQuantitative}
analogously.
\end{proof}

\subsection{Distributions with sub-exponential tail mass}

Random variables with sub-exponential\footnote{
Here we consider light-tailed random variables with sub-exponential tail mass
\citep{Vershynin_2018,Wainwright_2019}, in contrast
to heavy-tailed random variables with sub-exponential tail decay rates \citep{Foss_Korshunov_Zachary_2013,Nair_Wierman_Zwart_2022}.}
tail mass satisfying $\pr( \abs{X} \ge t ) \le 2e^{-ct}$
do not admit a convenient characterisation using a variance proxy.
Instead, they are typically characterised by the Orlicz norm $\senorm{X}$,
where we denote
\[
 \norm{X}_{\psi_p}
 \weq \inf \Big\{ K > 0 \colon \E e^{(\abs{X}/K)^p}  \le 2 \Big\}, \qquad p \ge 1.
\]
The Orlicz norm $\norm{X}_{\psi_2}$ characterises sub-Gaussian random variables:
$X$ is sub-Gaussian if and only if $\sgnorm{X} < \infty$.
Specifically, the norms $\sgnorm{X}$ and $\sigsg(X)$ are equivalent for centered variables $X$, with
sharp bounds given by $\sqrt{3/8} \sgnorm{X} \le \sigsg(X) \le \sqrt{\log 2} \sgnorm{X}$
\citep{Leskela_Zhukov_2026}. As a consequence of 
Proposition~\ref{the:SubGaussianComparisonQuantitative},
%Proposition~\ref{the:Variance},
it follows that the optimal sub-Poisson variance proxy of a centered variable $X$ is bounded by
\[
 \sigsp^2(X) \wle \log 2 \cdot \sgnorm{X}^2.
\]

Since Poisson distributions have lighter than exponential tails, one might expect
$\senorm{X}$ to be bounded in terms of $\sigsp^2(X)$.
While this is true, 
the bound is more subtle than in the sub-Gaussian case, because
$\sigsp(X)$ is not a norm (Example~\ref{exa:ScaledSkellam}).
Proposition~\ref{the:Orlicz1} below presents such a bound
in terms of
\updated{the principal branch of the Lambert \(W\) function,
defined by
\[
 W(x)e^{W(x)} = x,
 \qquad x \in [-e^{-1},\infty),
\]
or equivalently as the inverse of the map \(w \mapsto we^w\) on
\([-1,\infty)\).}
Because 
$W(x) \sim x$ for $x \to 0$,
and 
$W(x) \sim \log x$ for $x \to \infty$
\citep{Corless_etal_1996,Hoorfar_Hassani_2008} we see as 
a consequence of Proposition~\ref{the:Orlicz1} that
$\senorm{X} \lesim \frac{1}{\log(1/\sig)}$ for $\sig \to 0$,
and
$\senorm{X} \lesim \sig$ for $\sig \to \infty$.

\begin{proposition}
\label{the:Orlicz1}
For any centered sub-Poisson random variable $X$ with
variance proxy $\sig^2 > 0$,
\[
 \senorm{X}
 \wle 4 \left( \frac{1}{W(1/\sig)} \wedge \frac{1}{W(1/\sig^2)} \right).
\]
\end{proposition}

The proof of Proposition~\ref{the:Orlicz1} utilises the following simple lemma.

\begin{lemma}
\label{the:LambertW}
$c W(x) \le W(cx) \le W(x)$ for all $x \ge 0$ and $c \in [0,1]$.
\end{lemma}
\begin{proof}
Differentiation shows that $V(x) = x e^x$ is a strictly increasing bijection from $[0,\infty)$ to $[0,\infty)$.
Therefore, so is $W = V^{-1}$.
Denote $y_1 = W(cx)$ and $y_2 = W(x)$. Then
$y_1 e^{y_1} = cx$ and $y_2 e^{y_2} = x$,
and the inequality $y_1 \le y_2$ implies that
$y_1 = c e^{-y_1} x \ge c e^{-y_2} x = c y_2$.
\end{proof}

\begin{proof}[Proof of Proposition~\ref{the:Orlicz1}]
The assumptions imply that 
$M(\la) \le e^{\sig^2 \phi(\ala)}$ for all $\la \in \R$.
By applying Jensen's inequality and the fact that
$e^{\abs{t}} \le e^t + e^{-t}$,
it follows that for any $K > 0$ and $p>1$,
\[
 \E e^{\abs{X}/K}
 \wle \big( \E e^{p \abs{X}/K} \big)^{1/p}
 \wle \big( M(p/K) + M(-p/K) \big)^{1/p}
 \wle \big( 2 e^{\sig^2 \phi(p/K)} \big)^{1/p}.
\]
Hence $\E e^{\abs{X}/K} \le 2$ when
$e^{\sig^2 \phi(p/K)} \le 2^{p-1}$.
The latter condition is equivalent to $\phi(\frac{1+r}{K}) \le \tau^2 r$,
where $r=p-1$ and $\tau^2 = \frac{\log 2}{\sig^2}$.
We conclude that $\senorm{X} \le K$ for all $K > 0$ such that
\begin{equation}
 \label{eq:SENormUB1}
 \phi\Big(\frac{1+r}{K} \Big) \wle \tau^2 r
 \qquad \text{for some $r>0$.}
\end{equation}
We will combine \eqref{eq:SENormUB1} with two upper bounds of $\phi$
to derive tractable bounds for $\senorm{X}$.

(i)
Because $\phi(t) \le e^t-1$ for all $t \ge 0$, we see that
a sufficient condition for \eqref{eq:SENormUB1} is that
$\exp(\frac{1+r}{K}) \le 1 + \tau^2 r$ for some $r > 0$.
%$\frac{e^{p/K}-1}{p-1} \le \frac{\log 2}{\sig^2}$ for some $p > 1$.
This leads to
\begin{equation}
 \label{eq:SENormUB2}
 \senorm{X}
 \wle \inf_{r > 0} \frac{1+r}{\log(1 + \tau^2 r)}.
\end{equation}
To minimise the right side of \eqref{eq:SENormUB2}, denote $f(r) = \frac{1+r}{\log(1 + \tau^2 r)}$.
Then
\[
 f'(r)
 \weq \frac{1}{\log^2(1 + \tau^2 r)} \left( \log(1 + \tau^2 r) - \frac{1+r}{1 + \tau^2 r} \tau^2 \right)
 \weq \frac{h(\tau^2 r) - \tau^2}{(1 + \tau^2 r)\log^2(1 + \tau^2 r)},
\]
where $h(u) = (1+u)\log(1+u) - u$.  Differentiation shows that
$h$ is strictly increasing and hence invertible on $[0,\infty)$.
It follows that $f$ is minimised at $r_* = h^{-1}(\tau^2)/\tau^2$.
The equality $h(\tau^2 r_*) = \tau^2$ implies that
$\log(1+\tau^2 r_*) = \frac{1+r_*}{1+\tau^2 r_*} \tau^2$, which implies that
\[
 \inf_{r > 0} \frac{1+r}{\log(1 + \tau^2 r)}
 \weq f(r_*)
 \weq \frac{1 + h^{-1}(\tau^2)}{\tau^2}.
\]
The inverse of $h$ \updated{can be expressed as}
 $h^{-1}(y) = \exp( 1 + W(\frac{y-1}{e} )) -1 $.
Then
\[
 \frac{1 + h^{-1}(\tau^2)}{\tau^2}
 \weq \frac{\exp( 1 + W(\frac{\tau^2-1}{e} ))}{\tau^2}
 \wle \frac{\exp( 1 + W(\tau^2))}{\tau^2}.
\]
Because $W(\tau^2) e^{W(\tau^2)} = \tau^2$, we find that
$e^{1+W(\tau^2)} = e \tau^2 / W(\tau^2)$.
It follows by applying \eqref{eq:SENormUB2} and Lemma~\ref{the:LambertW} that
\begin{equation}
 \label{eq:SENormUB3}
 \senorm{X}
 \wle \frac{e}{W(\tau^2)}
 \wle \frac{e/(\log 2)}{W(1/\sig^2)}.
\end{equation}

(ii)
Alternatively, we may apply a bound
\[
 \phi(t)
 \weq \int_0^t \int_0^s e^r \, dr \, ds
 \wle \frac12 e^t t^2
 %\weq 2 e^t (t/2)^2
 \weq 2 ( (t/2) e^{t/2})^2
 \weq 2 V(t/2)^2,
\]
where $V(t) = t e^t$.
Hence for \eqref{eq:SENormUB1} it suffices that
$K > 0$ satisfies
$2 V( \frac{1+r}{2K})^2 \le \tau^2 r$ for some $r > 0$.
By plugging in $r=2$, this is equivalent to
$V(\frac{3/2}{K}) \le \tau$.
Because $W$ is the inverse of $V$, this is equivalent to
$K \ge \frac{3/2}{W(\tau)}$. We conclude with the help of Lemma~\ref{the:LambertW} that
\begin{equation}
 \label{eq:SENormUB4}
 \senorm{X}
 \wle \frac{3/2}{W(\tau)}
 \wle \frac{3/(2 \sqrt{\log 2})}{W(1/\sig)}.
\end{equation}
The claim follows by \eqref{eq:SENormUB3}--\eqref{eq:SENormUB4},
and noting that
$3/(2 \sqrt{\log 2}) \le e/(\log 2) \le 4$.
\end{proof}

\section{Examples}
\label{sec:Examples}

\begin{example}[Bernoulli distribution]
\label{exa:Bernoulli}
Assume that $\pr(X=1) = 1-\pr(X=0) = p$
for some $p \in (0,1)$.
By Proposition~\ref{the:SPBounded} (see also Remark~\ref{rem:BoundedCentered}),
$X$ is sub-Poisson with 
$\sigsp^2(X) = \sigspUL^2(X) = p(1-p)$.
Moreover \cite[Theorem 2.1]{Buldygin_Moskvichova_2013},
$X$ is sub-Gaussian with
$\sigsg^2(X) = \frac{1/2-p}{\log(1/p) + \log(1-p)}$
for $p \ne \frac12$.
In particular,
$\sigsp^2(X) \sim p$ is an order of magnitude smaller than
$\sigsg^2(X) \sim \frac{1}{\log(1/p)}$
in sparse regimes with $p \to 0$.
\end{example}

\begin{example}[Binomial distribution]
\label{exa:Binomial}
Let $X_n$ be distributed according to a binomial distribution with $n$ trials
and success probability $p$.
Then by Proposition~\ref{the:IndependentSum},
it follows that $\sigsp^2(X_n) \le n \sigsp^2(X_1)$. 
We saw in Example~\ref{exa:Bernoulli} that $\sigsp^2(X_1) = p(1-p)$.
Therefore, $\sigsp^2(X_n) \le n p(1-p)$. Because $\Var(X_n) = n p(1-p)$, it follows by
Proposition~\ref{the:Variance} that
$X_n$ is sub-Poisson with $\sigsp^2(X_n) = n p(1-p)$.
\end{example}

\begin{example}[Rademacher distribution]
\label{exa:Rademacher}
Assume that $\pr(X= \pm 1) = \frac12$.
Then by Proposition~\ref{the:SPBounded} (see also Remark~\ref{rem:BoundedCentered}),
we see that $X$ is sub-Poisson with $\sigsp^2(X) = \sigspUL^2(X) = 1$.
It is also well known that $\sigsg^2(X)=1$ \cite[Example 2.3]{Wainwright_2019}.
\end{example}

\begin{example}[Scaled Rademacher distribution]
Assume that $\pr(X_a = \pm a) = \frac12$ for some $a \ge 0$.
Because $\sigsg(X)$ is a norm on the space of centered random variables \cite[Theorem 1]{Buldygin_Kozachenko_1980},
we find that $\sigsg^2(X_a) = a^2 \sigsg^2(X_1)$.
In light of Example~\ref{exa:Rademacher}, we see that $\sigsg^2(X_a) = a^2$.
Because $\Var(X_a)=a^2$, we find that
\updated{$\Var(X_a) = \sigsp^2(X_a) = \sigsg^2(X_a)$}
by Propositions~\ref{the:Variance} and~\ref{the:SubGaussianComparisonQuantitative}.
We conclude that \updated{$X_a$} is sub-Poisson with \updated{$\sigsp^2(X_a) = a^2$}.
\end{example}

\begin{example}[Poisson distribution]
\label{exa:Poisson}
Let $X$ be Poisson distributed with parameter $a\geq 0$.
The centered moment generating function equals
$\E e^{\la (X-\E X)} = e^{a \phi(\la)}$. By noting that $\phi(-\la) \le \phi(\la)$ for all $\la \ge 0$
(Lemma~\ref{the:cosh}), we obtain $\E e^{\la (X-\E X)}\leq e^{a \phi(\abs{\la})}$.
By \eqref{eq:ProxySP}, we find that
$\sigsp^2(X) \leq a$. Since $\Var(X)=a$, we conclude that $\sigsp^2(X) = \sigspUL^2(X) = a$ by Proposition~\ref{the:Variance}.
On the other hand,
$X$ is not sub-Gaussian unless $a=0$ because $\phi(\la)/\la^2 \to \infty$ as $\la \to \infty$.
\end{example}

\begin{example}[Skellam distribution]
\label{exa:Skellam}
The \emph{Skellam distribution} \citep{Skellam_1946}
is defined as the law of a random variable 
$X = X_1-X_2$ where $X_1,X_2$ are independent and Poisson distributed with parameters $a_1,a_2 \ge 0$.
Then $\E X = a_1-a_2$, 
$\Var(X) = a_1+a_2$,
and
$\log \E e^{\la (X-\E X)} = a_1 \phi(\la) + a_2 \phi(-\la)$.
As in Example~\ref{exa:Poisson}, we find that
$X$ is sub-Poisson with
$\sigsp^2(X) = \sigspUL^2(X) = a_1+a_2$,
but not sub-Gaussian unless $a_1=a_2=0$.
\end{example}

\begin{example}[Scaled Skellam distribution]
\label{exa:ScaledSkellam}
Let $X_a = a X_1$ for some $a\geq 0$, where $X_1$ is distributed according to a Skellam distribution
with parameters $a_1=a_2=1$ (Example~\ref{exa:Skellam}).
Then $\E X_a = 0$ and $\log \E e^{\la X_a} = \phi(a\la) + \phi(-a\la)$.
Because $X_a$ is symmetric,
we see that $\sigsp^2(X_a) = \sigspUL^2(X_a)$. If $a>1$, then
\[
 \sigsp^2(X_a)
 \weq \sup_{\la > 0} \frac{\phi(a \la) + \phi(-a\la)}{\phi(\la)}
 \wge  \lim_{\la\to\infty} \frac{\phi(a\la)}{\phi(\la)}
 \weq \infty.
\]
If $a \le 1$, then  Propositions~\ref{the:Variance} and~\ref{the:Convexity}
imply $\Var(X_a)\leq\sigsp^2(X_a)\leq a^2\sigsp^2(X_1)$. As shown in Example~\ref{exa:Skellam},
we also know that $a^2\sigsp^2(X_1)=a^2\Var(X_1)=\Var(X_a)$. As a consequence,
$\sigsp^2(X_a)=\Var(X_a)=2a^2$. 
\end{example}

\begin{example}[Gaussian distribution]
\label{exa:Gaussian}
Let $X$ be a Gaussian random variable with variance $\sig^2$.
Because $\E e^{\la(X-\E X)} = e^{\sig^2 \la^2/2}$,
we see by \eqref{eq:ProxySG}--\eqref{eq:ProxySGU} that
$\sigsg^2(X) = \sigsgUL^2(X) = \Var(X)$.
By applying Propositions~\ref{the:Variance} and~\ref{the:SubGaussianComparisonQuantitative},
we find that
$\Var(X) \le \sigsp^2(X) \le \sigsg^2(X)$.
Therefore, $X$ is sub-Poisson with $\sigsp^2(X) = \Var(X)$.
\end{example}

\begin{example}[Exponential distribution]
\label{exa:Exponential}
Let $X$ be exponentially distributed with rate parameter $a > 0$.
Then
$\E e^{\la X} = \frac{1}{1-\la/a}$ for $\la < a$
and 
$\E e^{\la X} = \infty$ for $\la \ge a$.
Then $\sigspU^2(X) = \infty$ implies that $X$ is not upper sub-Poisson.
For the lower tail, we note that
\[
 \sigspL^2(X)
 \weq \sup_{\la > 0} \frac{\log \E e^{-\la(X-\E X)}}{\phi(\la)}
 \weq \sup_{\la > 0} \frac{\la/a - \log(1+\la/a)}{\phi(\la)}.
\]
Because $\log(1+t) \ge t - \frac12 t^2$ for all $t \ge 0$,
and $\phi(\la) \ge \frac12 \la^2$ for all $\la \ge 0$,
it follows that $\sigspL^2(X) \le 1/a^2$.
Because $\Var(X)=1/a^2$, we conclude by Proposition~\ref{the:Variance} that 
$X$ is lower sub-Poisson with $\sigspL^2(X) = 1/a^2$\updated{.}
\end{example}

\subsubsection*{Acknowledgements}

We are grateful to Maximilien Dreveton for pointing out that the statements in Proposition~\ref{the:SPBounded}:(iii)--(iv) hold without the assumption $\E X = 0$. \updated{We also thank the anonymous reviewer for helpful comments that improved the presentation.}

Ian Välimaa's research was partly supported by a doctoral research grant from the Emil Aaltonen Foundation.

%\begin{appendices}
\appendix

\section{Auxiliary results}
\label{sec:AuxiliaryResults}

\begin{lemma}
\label{the:cosh}
For all \( x \in \mathbb{R} \),
the function $\phi(x) = e^x - 1 - x$ satisfies
\[
 \phi(x) \wle \phi(\abs{x})
 %e^{{x}} - {x} - 1 \wle e^{\abs{x}} - \abs{x} - 1
\]
and
\[
 \cosh(x)-1 \wle \phi(\abs{x}) \wle 2( \cosh(x)-1).
 %\cosh(x)-1 \wle e^{\abs{x}} - \abs{x} - 1 \wle 2( \cosh(x)-1).
\]
\end{lemma}
\begin{proof}
Denote $ \psi(x) = \cosh(x) - 1 $.
The inequality \( \phi(x) \le \phi(|x|) \) holds trivially as identity for
$x \ge 0$. For $x < 0$, note that
\[
 \phi(|x|) - \phi(x)
 \weq \phi(-x) - \phi(x)
 \weq e^{-x} - e^{x} + 2x
 \wge 0,
\]
by convexity of the exponential.

For the upper bound of $\phi(\abs{x})$, assume \( x \ge 0 \), and note that
\[
 2\psi(x) - \updated{\phi(x)}
 \weq e^{-x} + x - 1
 \wge 0,
\]
since \( e^{-x} \ge 1 - x \).
For the lower bound of $\phi(\abs{x})$, again with \( x \ge 0 \), we see that
\[
 \phi(x) - \psi(x)
 \weq \tfrac{1}{2}\bigl(e^{x} - e^{-x} - 2x \bigr)
 \wge 0,
\]
because the function \( h(x) = e^x - e^{-x} - 2x \) satisfies \( h(0) = h'(0) = 0 \) and
\[
% h''(x) = e^{x} + e^{-x} > 0 \quad \text{for } x > 0,
h''(x) = \updated{e^{x} - e^{-x}} > 0 \quad \text{for } x > 0,
\]
implying \( h(x) \ge 0 \). 
We conclude that $\psi(x) \le \phi(|x|) \le 2 \psi(x)$ holds for all $x \ge 0$.
The same bounds hold for $x < 0$ due to $\psi(-x) = \psi(x)$.
\end{proof}

\bibliographystyle{abbrv}
\bibliography{lslReferences}

\end{document}